\documentclass[10pt,reqno]{amsart}
\usepackage{palatino}
\usepackage{amsmath,amssymb,mathrsfs}

\newtheorem{thm}{Theorem}
\newtheorem{lemma}[thm]{Lemma}

\newtheorem{prop}[thm]{Proposition}
\newtheorem*{thm A}{ Theorem A}
\newtheorem*{thm B}{ Theorem B}

\newcommand{\N}{{\mathbb N}}
\newcommand{\Z}{{\mathbb Z}}
\newcommand{\C}{{\mathbb C}}

\newcommand{\R}{{\mathbb R}}

\newcommand{\T}{{\mathbb T}}
\newcommand{\D}{{\mathbb D}}

\newcommand{\be}{\mathcal{B}}
\newcommand{\Al}{\mathcal A}
\newcommand{\El}{\mathcal E}
\newcommand{\FF}{\mathcal F}
\newcommand{\Gl}{\mathcal G}
\newcommand{\HH}{\mathcal H}

\newcommand{\Sl}{\mathcal S}
\newcommand{\Ml}{\mathcal M}
\newcommand{\Zl}{\mathcal Z}

\newcommand{\DD}{\mathcal D}

\begin{document}

\title[Algebraic properties of Toeplitz operators on  generalized Fock spaces on $\C^d$]{Algebraic properties of Toeplitz operators on  generalized Fock spaces on $\C^d$}

\author{H. Bommier-Hato} 

\address{Bommier-Hato: Faculty of Mathematics,
University of Vienna,
Oskar-Morgenstern-Platz 1, 1090 Vienna, Austria } \email{helene.bommier@gmail.com}

\thanks{The author was supported by the FWF project P 30251-N35.}



\maketitle

 
\begin{abstract} 
We study two problems involving algebraic properties of Toeplitz operators on generalized Fock spaces on $\C^d$ with  weights of the form $\left|z\right|^{2s} e^{-\left|z\right|^{2m}}$, $m\geq 1,\ s\geq 0$. We determine the commutant of a given Toeplitz operator with a radial symbol which satisfies certain growth conditions. We  also discuss the equation $T_fT_g=0$, when $f$ or $g$ is radial.

\end{abstract}

\section{Introduction}
In this paper, we address two questions related to algebraic properties of Toeplitz operators, namely the \textit{commuting  problem} and the \textit{zero product problem}.

The question of characterizing the symbols $f$ and $g$ such that 
	\begin{equation}\label{toep com}
		T_fT_g=T_gT_f
	\end{equation}
	appears in various contexts. Apart from being of mathematical interest, it is related to questions from Quantum Physics, such as the construction of spectral triples \cite{englis2}. On the Hardy space $H^2(\mathbb{T})$, a complete solution was given 
	by Brown and Halmos in \cite{brow-halm}, where a characterization of the 
	bounded functions $f$ and $g$ satisfying (\ref{toep com}) is provided. On the unweighted Bergman space $A^2(\D)$, the problem is still open for general symbols. However,  Ahern and \v Cu\v ckovi\' c \cite{AC} treated the case of 
	harmonic symbols, and later, \v Cu\v ckovi\' c and Rao \cite{CuR} showed that if $f$ is a bounded nonconstant radial symbol, and $g$ is bounded such that (\ref{toep com}) holds, then $g$ is also radial. Their result has been generalized for the unit ball of $\C^d$ in \cite{tLe}.
	
	The so-called "zero-product problem" for Toeplitz operators on the Hardy space $H^2(\D)$ can be stated as follows: \\
	If $f_1,\cdots,f_N$ are bounded functions, does $T_{f_1}\cdots T_{f_N}=0$ imply that one of the symbols is zero?
	
	The problem has been solved in the affirmative for $N=2$ in \cite{brow-halm} and recently for general $N$ by Aleman and Vukoti\'c \cite{aleman-vukotic}. For Bergman spaces, some partial results have been obtained  (see for instance \cite{AC}, \cite{tLe1}).

In the context of entire functions, the most classical framework is the Segal-Bargmann space $F^{2}_{\alpha}$ of Quantum Mechanics, which is the space of all entire functions $f:\C^d\rightarrow\C$, $d\geq 1$, such that $\left|f\right|^2$ is integrable with respect to the
Gaussian 
$$d\mu_{\alpha}(z):=\left(\frac{\alpha}{\pi}\right)^d e^{-\alpha\left|z\right|^2}dv(z),$$
where $dv(z)$ stands for the Lebesgue volume on $\C^d$ and $\alpha>0$ \cite{Foll,JPR,Zf}. 
 
 Given real numbers $m\geq 1$, $\alpha>0$, $s\geq 0$, we consider the probability measure on $\C^d$,
\begin{equation}\label{dmu cm alpha}
d\mu_{m,\alpha,s}(\zeta):=c_{m,\alpha,s}\left|\zeta\right|^{2s}e^{-{\alpha}\left|\zeta\right|^{2m}}dv(\zeta),\ \text{where  }
	\ c_{m,\alpha,s}=\frac{m\alpha^{\frac{d+s}{m}}}{\pi^d}\frac{\Gamma(d)}{\Gamma\left(\frac{d+s}{m}\right)}. 
\end{equation}

 We denote by $L^{2}_{m,\alpha,s}$ the space of   measurable functions $f:\C^d \rightarrow\C$ such that
\begin{gather*}
	\left\|f\right\|_{L^{2}_{m,\alpha,s}}^2:=\int_{\C^d}\left|f(\zeta) \right|^2d\mu_{m,\alpha,s}(\zeta)<\infty,
\end{gather*}
and  $F^{2}_{m,\alpha,s}$ is the space of  entire functions which are in $L^{2}_{m,\alpha,s}$, equipped with the same norm.\\

Endowed with the inner product, 
\begin{equation*}
 \left\langle x,y\right\rangle_{m,\alpha,s}:= \int_{\C^d} x(\zeta)\overline y(\zeta) d\mu_{m,\alpha,s}(\zeta),\ x,y\in F^{2}_{m,\alpha,s},
\end{equation*}  
$F^{2}_{m,\alpha,s}$ is a Hilbert space.  
The orthogonal projection $P_{m,\alpha,s}:L^{2}_{m,\alpha,s}\rightarrow  F^{2}_{m,\alpha,s}$ is defined by
\begin{equation*}
	P_{m,\alpha,s}x(z)=\int_{\C^d}x(\zeta)K_{m,\alpha,s}(z,\zeta)d\mu_{m,\alpha,s}(\zeta),\ x\in L^{2}_{m,\alpha,s},\  z\in\C^d,
\end{equation*}
 where  $K_{m,\alpha,s}$ denotes the  reproducing kernel in $F^{2}_{m,\alpha,s}$.\\

 For a suitable function $f:\C^d\rightarrow\C$, the Toeplitz operator $T_f$, of symbol $f$ is defined on $F^{2}_{m,\alpha,s}$ by 
$$T_f=P_{m,\alpha,s}M_{f}|_{F^{2}_{m,\alpha,s}} ,$$
  $M_f$ being the operator of multiplication by $f$. The commutant of a radial Toeplitz operator has been studied for the Segal-Bargmann space $(m=1,s=0)$ in \cite{BauerLee,BauerLe} and for Fock-Sobolev spaces  $(m=1,s\text{ real })$ in \cite{choe}.  On the Fock space, the equation 
		\begin{equation*}\label{N=2}
		T_fT_g=0
	\end{equation*}
	has been considered in \cite{BauerLe}, when  one of the functions is  radial. 

In this paper,  we extend the results of \cite{BauerLee,BauerLe,choe} to the generalized Fock space $\FF:=F^{2}_{m,1,s}$, $m\geq 1,s\geq 0$, the parameter $\alpha$ being irrelevant for these questions.  The inner product $\left\langle .,.\right\rangle_{m,1,s}$ will simply be denoted by $\left\langle .,.\right\rangle_{\FF}$.

The growth of the function $f$ plays an important role in the study of the commutant of $T_f$. We shall consider the class of symbols $\Sl(\C^d)$, which consists of functions $f:\C^d\rightarrow \C$ such that $z\mapsto f(z) e^{-c\left|z\right|^{2m}}$ is bounded for all $c>0$. 

We first describe the commutant of $T_f$, when $f$ is a radial function in $\Sl(\C^d)$.

\begin{thm A}\label{thm A}
Let $f,g$ be in $\Sl(\C^d)$, $f$ being a non constant radial function. Then $T_f T_g= T_g T_f$ on the set of analytic polynomials if and only if $g(e^{i\theta} z)=g(z)$, for a.a. $\theta\in\R$ and $z\in \C^d$.
\end{thm A}

We also give a counterexample when $f\notin\Sl(\C^d)$.\\

Our second result is about the zero product problem.
\begin{thm B}\label{thm B}
Let $f,g$ be in $\Sl(\C^d)$, $f$ being  radial. Suppose  $T_f T_g=0$ or $ T_g T_f=0$ on the set of analytic polynomials. Then either $f=0$ or $g=0$ a.e..
\end{thm B}

Following the ideas of \cite{BauerLe}, we study a more general question. For given radial functions $f_1,f_2$ in $\Sl(\C^d)$, we aim to find the solutions $g$ of the equation 
\begin{equation}\label{T f1g=gf2}
	T_g T_{f_1}=T_{f_2}T_g.
\end{equation}

However, the situation  when $m>1$ is more involved than the case of the Segal-Bargmann ($m=1$).

 When $m=1$, the reproducing kernels are exponential functions, which gives  exact formulas (see  \cite{Zf}).  When $m>1$, there is no simple expression for the reproducing kernel of $F^{2}_{m,1,s}$. Therefore, we have to develop new techniques, based on the    asymptotics of kernel functions.

For the Gaussian measure, it is used that for a multinidex $\nu, $ the moment $S_{1,0}(\nu)$, is a polynomial in $\nu$, and that the weight satisfies the multiplicative relation 
$$  e^{-\left|z\right|^2}=e^{-\left|z_1\right|^2}\times e^{-\left|z'\right|^2},\ \text{ for }z=(z_1, z')\in\C^d.$$
This is no longer true for general $m>1$. On another hand, relation (\ref{T f1g=gf2}) leads to an equation involving the eigenvalues of $T_{f_1}$ and $T_{f_2}$, which is more complicated when $m>1$ (see section \ref{eq Toep} for details). \\

Throughout the paper, $\N$ is the set of all non negative integers, $\Pi $ is the half plane $\Re\zeta>0,$
and 
$\chi_A$ is the indicator function of the set $A$.
 The euclidean norm on $\C^d$ is denoted simply by $\left|.\right|$, and the standard inner product by $\left\langle .,.\right\rangle$.
 For two functions $f, g$, the notation $f=O(g)$ or $f \lesssim g$,  means that there exists a constant $C$ such that $f\leq C g$ .
 If  $f=O(g)$ and $g=O(f)$, we  write $f\asymp g$.\\



The paper is organized as follows. In section \ref{prelim}, we present general properties of the space $F^{2}_{m,\alpha,s}$. We also introduce an adequate class of symbols, for which a finite product of Toeplitz operators is defined on   $\FF$. In section \ref{eq Toep}, we  sudy the equation $T_{f_1}T_g=T_gT_{f_2}$, for radial functions $f_1,f_2$ in $\Sl(\C^d)$. Using the Mellin transform, we write conditions for $T_{f_1}T_g$ and $T_gT_{f_2}$ to coincide on holomorphic polynomials. Applying the results from section \ref{eq Toep} to the  situation when $f_1=f_2$, we deal with the commutant problem in  Section \ref{comm}. The zero product problem is considered in Section   \ref{zeroPb}.

\section{Preliminaries} \label{prelim}


\subsection{Fock-type spaces}

Since the weight  $\left|\zeta\right|^{2s}e^{-\alpha\left|\zeta\right|^{2m}}$ depends only on $\left|\zeta\right|$, 
the monomials $\left(v_{\nu}\right)_{\nu\in\N^d}$, with $v_{\nu}(\zeta)=\zeta^{\nu}$, form an orthogonal basis in $ F^{2}_{m,\alpha,s} $.   Here, for a multiindex $\nu=(\nu_1,\cdots,  \nu_d)\in\N^d$, and $\zeta\in\C^d$, we use  the standard  notation $\nu!=\nu_1!\cdots\nu_d!$, $\left|\nu\right|=\nu_1+\cdots+\nu_d$ and $\zeta^{\nu}=\zeta_1^{\nu_1}\cdots\zeta_d^{\nu_d}$. Integration in spherical  coordinates gives that
\begin{equation}
\label{s alpha}
	 S_{\alpha,s}(\nu):=\left\|\zeta^{\nu}\right\|^{2}_{F^{2}_{m,\alpha,s}}= C^{-1}_{s}  \frac{\nu!\Gamma\left(\frac{d+s+\left|\nu\right|}{m}\right)}{\Gamma\left(d+\left|\nu\right|\right)\alpha^{\frac{\left|\nu\right|}{m}}}, \text{ where   }  C_s:=\frac{\Gamma\left(\frac{d+s}{m}\right)}{\Gamma(d)}.
\end{equation}
Then the set $\be=\left\{e_{\nu}(z)=\left[ S_{\alpha,s}(\nu)\right]^{-1/2}z^{\nu},\ \nu\in\N^d\right\}$ is an orthonormal basis for $ F^{2}_{m,\alpha,s} $.
We use  the  theory from Aronszajn \cite{Ar} to compute the reproducing kernel of $ F^{2}_{m,\alpha,s} $. It is given by
\begin{align*}
\label{kernel Fm}
	K_{m,\alpha,s}(\xi,\zeta)& =\sum_{\nu}\frac{\xi^{\nu}\overline{\zeta}^{\nu}}{S_{\alpha,s}(\nu)	} =\sum^{+\infty}_{k=0}\frac{\left(\alpha^{1/m}\left\langle \xi,\zeta\right\rangle\right)^k}{k!}\frac{\Gamma(d+k)}{\Gamma\left(\frac{d+s+k}{m}\right)} \nonumber\\
		&=C_{s} E^{(d-1)}_{\frac{1}{m},\frac{1+s}{m}}\left(\alpha^{1/m}\left\langle \xi,\zeta\right\rangle\right),\quad  \text{ for }\xi,\zeta\in\C^d,
\end{align*}
where 
\begin{equation*}
E_{\beta,\gamma}(z)=\sum^{+\infty}_{k=0}\frac{z^k}{\Gamma\left(\beta k+\gamma\right)},\  \beta,\gamma>0,
\end{equation*}
is the Mittag-Leffler function. Recall that
the  orthogonal projection $P_{m,\alpha,s}: L^{2}_{m,\alpha,s}\rightarrow F^{2}_{m,\alpha,s}$ 
 is defined by
\begin{equation*}
P_{m,\alpha,s}f(\zeta)=\int_{\C^d}K_{m,\alpha,s}(\zeta,\xi)f(\xi)d\mu_{m,\alpha,s}(\xi),\quad \zeta\in\C^d.
\end{equation*}

 In order to estimate the growth of $K_{m,1,s}(x,y)$,  we shall  use
  asymptotics of  Mittag-Leffler function and its derivatives. The asymptotic expansion of $E_{\beta,\gamma}(z)$ is based on the integral representation 
	\begin{equation}\label{integral ML }
	E_{\beta,\gamma}(z)=\frac{1}{2\pi i}\int_{\mathcal C}\frac{t^{\beta-\gamma}e^t}{t^{\beta}-z}dt,\ \Re\beta>0,\ \Re\gamma>0,\ z\in \C,
	\end{equation}
	where the path of integration $\mathcal C$ is a loop starting at $-\infty$, encircling the disk $\left|t\right|\leq\left|z\right|^{\frac{1}{\beta}}$ counterclockwise, and ending at $-\infty$, with $|\arg t|<\pi$ for $t\in \mathcal C$ (see e.g. Bateman and Erdelyi \cite{Bo}, vol. III, ?8.1, formulas (21)?22);  Wong and Zhao \cite{WoZh}). 
	In particular, we have
\begin{equation*}\label{asymp ML m>1/2}
	E_{\frac{1}{m},\frac{1+s}{m}}(z)=\begin{cases}
	mz^{m-1-s}e^{z^m}+O\left(z^{-1}\right),\ \left|\arg z\right|\leq\frac{\pi}{2m},\\
	O\left(\frac{1}{z}\right),\ \frac{\pi}{2m}<|\arg z|<\pi
	\end{cases}
\end{equation*}
for $m > \frac{1}{2}$ and  $z\in\C\setminus\left\{0\right\}$.  

The expansion can be differentiated termwise any number of times, for example by derivation of the integral expression (\ref{integral ML }) (see \cite{Bo} for details in the case $ \beta=\gamma  $). For a positive integer $l$, we get
\begin{equation*}\label{def pk}
	\frac{d^{l-1}}{dz^{l-1}}(mz^{m-1-s}e^{z^m})\asymp z^{l(m-1)-s}e^{z^m},\ \text{ as }\left|z\right|\rightarrow \infty,\ \left|\arg z\right|\leq\frac{\pi}{2m}.
\end{equation*}

 For positive real $z$,  we then obtain the following asymptotics, 
\begin{equation*}\label{asymp deri ML m > 1/2}
	E^{(d-1)}_{\frac{1}{m},\frac{1+s}{m}}(z)=O\left(z^{d(m-1)-s}e^{z^m}\right),\ \text{ as }z\rightarrow \infty.
	\end{equation*}
Notice that  $\left|E^{(d-1)}_{\frac{1}{m},\frac{1+s}{m}}(z) \right|\leq E^{(d-1)}_{\frac{1}{m},\frac{1+s}{m}}(\left|z\right|)$, because  the power series of $E^{(d-1)}_{\frac{1}{m},\frac{1+s}{m}}$ has positive coefficients. It follows that
\begin{equation}\label{asymp kernel}
\left|K_{m,\alpha,s}(x,y)\right|\lesssim \left|\left\langle x,y\right\rangle\right|^{d(m-1)-s}e^{\alpha \left|\left\langle x,y\right\rangle\right|^m} \lesssim\left(\left|x\right|\left|y\right|\right)^{d(m-1)-s}e^{\alpha\left|x\right|^m\left|y\right|^m},
\end{equation}
as $\left|x\right|\left|y\right|\rightarrow\infty$.	\\

Let us now present results of density, which are useful for our study. 
Denote by $\C\left[z, \overline{z}\right]$ the space of polynomials  in the variables $z=(z_1,\cdots,z_d)$ and $\overline z=(\overline z_1,\cdots,\overline z_d)$, and by 
 $\C\left[z\right]$ the space of holomorphic polynomials. In the case of the standard Gaussian measure, it is known that the set $\C\left[z, \overline{z}\right]$ is dense in $L^2\left(\C^d, \pi^{-d}e^{-\left|z\right|^{2}}dv(z)\right)$. The following Proposition, which states a property of $L^p$ spaces with respect to an exponentially bounded measure, is certainly known to specialists. However, we give it for completeness.

\begin{prop}\label{density poly expo bd}
Let $N$ be a positive integer, and  $\mu$ be a Radon measure on  $\R^N$, having moments of all orders. Suppose that,  for some $a>0$,
$$ \int_{\R^N} e^{a\left|x\right|}d\mu(x)<\infty.$$
Then, for all $p\in[1,+\infty[$, the set of all polynomials in the real variables  $x_1,\cdots,x_N$ is dense in $L^p(\R^N, d\mu)$.
 \end{prop}

\begin{proof}
We briefly extend the argument provided in \cite{Berg} to the multidimensional setting. Denote by $p'$ the conjugate exponent of $p$ , i.e. $\frac{1}{p}+\frac{1}{p'}=1$. Let $f$ be in $L^{p'}(\R^N, d\mu)$, such that 
\begin{equation*}\label{integ f=0}
	\int_{\R^N}f(x)x^k d\mu(x)=0,\ \text{ for all }k\in\N^N.
\end{equation*}
For $x=(x_1,\cdots,x_N)\in\R^N$, and $z=(z_1,\cdots,z_N)\in\C^N$, set $x.z=\sum^{N}_{1}x_jz_j$. Now, consider the function defined by
$$  F(z):=\int_{\R^N}f(x)e^{ix.z} d\mu(x),\ z=(z_1,\cdots,z_N)\in\C^N.$$
Using H\"older inequality and Cauchy-Schwarz  inequality, we observe that
\begin{align*}
\int_{\R^N}\left|f(x)e^{ix.z} \right|e^{\frac{a}{2p}\left|x\right|}d\mu(x)&= \int_{\R^N}\left|f(x) \right|e^{-\sum^{N}_{1}x_j \Im z_j}e^{\frac{a}{2p}\left|x\right|} d\mu(x)\\
&\leq \left\|f\right\|_{L^{p'}(\R^N, d\mu)}
\left(\int_{\R^N} e^{ p\left|x\right|\left|\Im z\right|}  e^{\frac{a}{2}\left|x\right|} d\mu(x)\right)^{1/p}.
\end{align*}
By the Dominated Convergence Theorem,  Morera's Theorem and Osgood's Theorem, $F$ is holomorphic in $\Omega=\left\{z=(z_1,\cdots,z_N)\in\C^N,\ \left|\Im z\right|< \frac{a}{2p}\right\}$. Its partial derivatives are expressed by 
$$\partial^{\kappa}_zF(z)= \int_{\R^N}\left(ix\right)^{\kappa}f(x)e^{ix.z} d\mu(x),\ \text{ for  }\kappa\in\N^N.$$
The  assumption induces that $\partial^{\kappa}_zF(0)=0$,  for every multiindex $\kappa$. By the Identity Principle, we have  $F\equiv 0$. In particular, the Fourier transform
$$  F(\xi)=\int_{\R^N}f(x)e^{ix.\xi} d\mu(x),\ \xi\in\R^N$$
is identically zero in $\R^N.$ Therefore $f=0$ a.e..
\end{proof}
Since $\C\left[z, \overline{z}\right]$ is the set of all polynomials in the variables $\Re z$ and $\Im z$, we obtain the following 

\begin{prop}\label{density poly }
Let $m\geq 1,\ \alpha>0, \ s\geq 0$. The set $\C\left[z, \overline{z}\right]$ is dense in  $L^2\left(\C^d,d\mu_{m,\alpha,s}\right)$.

\end{prop}

Moreover,  due to the orthogonality of monomials,  the holomorphic polynomials are dense in $F^{2}_{m,\alpha,s}$. 

\subsection{Toeplitz operators on $\FF$}
Recall that $\FF$ denotes the space $F^{2}_{m,1,s}$. For a measurable symbol $u$, the Toeplitz operator of symbol $u$ is defined on 
$\text{Dom} (T_u)=\left\{h\in\FF,\ uh\in L^2\left(\C^d,d\mu_{m,1,s}\right)\right\}$ by 
$$  T_uh=P_{m,1,s}(uh).$$
Thus, the reproducing property in $\FF$  induces that 
\begin{equation*}\label{toep integral}
	T_u h(z)=\int_{\C^d} u(x) h(x) K_{m,1,s}(z,x)d\mu_{m,1,s}(x),\ \text{ for } z\in\C^d.
	\end{equation*}
 Moreover, if $f$ is in $\FF$ and $h$ is in $\text{Dom} (T_u)$, we have 
\begin{equation*}\label{toep bilin}
\left\langle 	T_u h,f\right\rangle_{\FF}=\left\langle P_{m,1,s}(uh),f\right\rangle_{\FF}=\left\langle uh,f\right\rangle_{\FF}.
\end{equation*}

We now define the space of measurable functions which have at most polynomial growth at infinity,
$$\Al(\C^d):=\left\{u:\C^d\rightarrow \C,\ \text{measurable },\ \exists C,c>0,\ \left|u(z)\right|\leq C\left(1+\left|z\right|\right)^c \text{ a.e.}\right\}.  $$

When dealing with radial symbols, we will often identify  radial functions $f$ on $\C^d$  (i.e. $f= f\circ \left|.\right|$),  with functions on $\R_+$ in the obvious manner. In particular, we shall consider the space (see \cite{BauerLee}) $ \Al(\R_+)$ of measurable functions $u:\R_+\rightarrow \C,$ for which there exist $\eta,\rho>0,$ such that 
$$ \left|u(x)x^{-\eta}\right|=O(1)\text{ and }\left|u\left(\frac{1}{x}\right)x^{-\rho}\right|=O(1),\ \text{ for a.a. }x\geq 1. $$

For $c\geq 0$, we also define the Banach space 
$$ \DD_c:=\left\{u:\C^d\rightarrow \C,\ \text{measurable },\ \exists C>0,\ \left|u(z)\right|\leq C e^{c\left|z\right|^{2m}}\text{ a.e. }\right\}, $$
equipped with the norm $\left\|u\right\|_{\DD_c}:=\left\|u e^{-c\left|.\right|^{2m}}\right\|_{L^{\infty}(\C^d)}.$ If the symbol $u$ is in $ \DD_c$ for some $c<1/2$, then all spaces $ \DD_c'$ with $c+c'<1/2$ are contained in the domain of $T_u$. In particular, holomorphic polynomials are in $\text{Dom} (T_u)$ and $T_u$ is densely defined on $\FF$.

For  certain classes of symbols, the correspondence $g\rightarrow T_g$ is one-to-one (see \cite{Foll} for the classical case).

\begin{prop}\label{one to one }
Let $0<c<\frac{1}{2}$. If $g$ is in $\DD_c$, and $T_g=0$ on $\C\left[z\right]$, then $g=0$ a.e..
\end{prop}

\begin{proof}
The hypothesis entails that
 $$  \left\langle T_g x^{\kappa} ,x^{\nu} \right\rangle_{\FF}=\int_{\C^d}g(x) x^{\kappa} \overline x^{\nu}d\mu_{m,1,s}(x)=0,\ \text{ for all multiindices  } \kappa,\nu. $$
 Since $\C\left[z,\overline z\right]$, which is the linear span of $\left\{z^{\kappa}\overline z^{\nu}\right\}_{\kappa,\nu\in \N^d}$, is dense  in  $L^2\left(\C^d,d\mu_{m,1,s}\right)$ (Proposition \ref{density poly }), we see that $g=0$ a.e..
\end{proof}

\subsection{Products of Toeplitz operators}
In our framework,  products of Toeplitz operators are possibly unbounded. Thus, we need to ensure that these products are densely defined.   We construct an appropriate algebra of Toeplitz operators, in the spirit of \cite{Bauer}. The space of symbols 
$$ \Sl(\C^d):=\cap_{c>0}\DD_c $$
 is a $*$-algebra under pointwise multiplication and with respect to complex conjugation.

 For $j\in\N$, we set  $c_j=1/2-1/(2j+2)$, which  satisfies
\begin{equation}\label{c j+1}
	c_{j+1}=\frac{1}{4(1-c_j)}.
	\end{equation}
  and define
$$  \HH_j:=\DD_{c_j}\cap \FF.$$
Notice that the space of holomorphic polynomials is contained in each $  \HH_j$.
The norm on $  \HH_j$ is given by
$$  \left\|f\right\|_{\HH_j}=\left\|\exp(-c_j\left|.\right|)f\right\|_{L^{\infty}(\C^d)}.$$

Hence, we  obtain a scale of Banach spaces
$$  \C\backsimeq \HH_0\subset \HH_1 \subset \cdots \subset\HH_j\subset\HH_{j+1}\subset \cdots \subset\HH:=\cup_{j\in\N} \HH_j\subset   \FF,$$
the inclusion $\HH\subset   \FF$ being dense for the topology of $\FF$.

\begin{prop}\label{projection Dcj Dcj+2 }
Let $j\in\N$. 
\begin{enumerate}
	\item [(a)] The projection $P_{m,1,s}$ acts boundedly from $\left(\DD_{c_j},\left\|.\right\|_{\DD_{c_j}}\right)$ to $\left(\DD_{c_{j+2}},\left\|.\right\|_{\DD_{c_{j+2}}}\right)$. For $f\in\DD_{c_j}$, we have
$$  \left\|P_{m,1,s}f\right\|_{\DD_{c_{j+2}}}
\lesssim\left(1-c_j\right)^{\epsilon}\left\|f\right\|_{\DD_{c_j}}$$
for some constant $\epsilon$.

	\item [(b)] If $u\in\Sl(\C^d)$, then $M_u$,  the multiplication by $u$,  acts boundedly from   $\left(\DD_{c_j},\left\|.\right\|_{\DD_{c_j}}\right)$ to $\left(\DD_{c_{j+1}},\left\|.\right\|_{\DD_{c_{j+1	}}}\right)$.
\end{enumerate}

\end{prop}
In order to prove Proposition \ref{projection Dcj Dcj+2 }, we take the following computational Lemma from \cite{BEY}.

\begin{lemma}\label{l8 BEY}
For $\eta>-2d$, $B>0,$ and $A>0$,
$$ \int_{\C^d}\left|x\right|^{\eta} e^{B\left|x\right|^m}e^{-A\left|x\right|^{2m}}dv(x)\lesssim B^{\frac{\eta+2d}{m}-1}e^{{B^2}/{4A}},$$
as $B\rightarrow\infty.$
\end{lemma}

\begin{proof}[Proof of Proposition \ref{projection Dcj Dcj+2 }]
Put $b=1-c_j$,  $\gamma=d(m-1)-s$, and
take $z\in\C^d$. From the integral formula
$$  P_{m,1,s}f(z)=\int_{\C^d}f(x)K_{m,1,s}(z,x) d\mu_{m,1,s}(x),$$
 we get
\begin{align*}
  \left|P_{m,1,s}f(z)\right|&\lesssim \left\|f\right\|_{\DD_{c_j}}\int_{\C^d}\left|K_{m,1,s}(z,x)\right|\left|x\right|^{2s}e^{-b\left|x\right|^{2m}}dv(x).
	\end{align*}
	Now let $\left|z\right|$ tend to $\infty$. 
	
	First assume $\gamma\geq 0$.
Using a change of variable and  estimate (\ref{asymp kernel}),
 we obtain
$$ \left|P_{m,1,s}f(z)\right| \lesssim \left\|f\right\|_{\DD_{c_j}} b^{-\frac{d+s}{m}}\int_{\C^d}\left(b^{-\frac{1}{2m}}\left|z\right|\right)^{\gamma}\left|y\right|^{\gamma+2s}e^{b^{-\frac{1}{2}}\left|y\right|^{m}\left|z\right|^{m}}e^{-\left|y\right|^{2m}}dv(y).$$
We next use  Lemma \ref{l8 BEY} with $A=1$ and $B=b^{-\frac{1}{2}}\left|z\right|^m$, combined with (\ref{c j+1}). Thus
\begin{align*}
  \left|P_{m,1,s}f(z)\right|&\lesssim \left\|f\right\|_{\DD_{c_j}}b^{-\frac{d+s}{m}}B^{\frac{\gamma}{m}}  B^{\frac{\gamma+2s+2d}{m}-1}e^{\frac{\left|z\right|^{2m}}{4b}}\\
	&\lesssim \left\|f\right\|_{\DD_{c_j}}b^{\epsilon }\left|z\right|^{m(2d-1)}e^{c_{j+1}{\left|z\right|^{2m}}}.
	\end{align*}
	Similarly, when $\gamma\leq 0$,
	\begin{align*}
 \left|P_{m,1,s}f(z)\right|&\lesssim
 \left\|f\right\|_{\DD_{c_j}}b^{-\frac{d+s}{m}}  \int_{\C^d}\left|y\right|^{2s}e^{b^{-\frac{1}{2}}\left|z\right|^{m}\left|y\right|^{m}}e^{-\left|y\right|^{2m}}dv(y)\\
	&\lesssim \left\|f\right\|_{\DD_{c_j}}b^{-\frac{d+s}{m}}B^{\frac{2s+2d}{m}-1}e^{\frac{ \left|z\right|^{2m}}{4b}}\\
	&\lesssim \left\|f\right\|_{\DD_{c_j}}b^{\epsilon}\left|z\right|^{2s+2d-m}e^{c_{j+1}{\left|z\right|^{2m}}}.
		\end{align*}
	The definition of $\DD_{c_{j+2}}$ gives $(a)$.\\
	
	$(b)$ is a direct consequence of the definition of the spaces $(\DD_{c_{j}})_j$ and $\Sl(\C^d)$, altogether with the fact that $c_{j+1}> c_j$.\\
\end{proof}
The above Proposition ensures that a Toeplitz operator $T_u$ with symbol $u$ in $\Sl(\C^d)$ maps $\HH$ into $\HH$. In particular, if $N$ is  a positive integer, and  $u_1,\cdots, u_N$ are in $\Sl(\C^d)$, the Toeplitz product $T_{u_1}\cdots T_{u_N}$ is densely defined on $\FF$.

\section{An equation for Toeplitz operators}\label{eq Toep}
As observed in  \cite{BauerLe},  the commuting problem and the zero product problem are special cases of the equation
$$ T_{f_1} T_{g} =T_g  T_{f_2},\ \quad\quad\quad\quad\quad\quad{\El(f_1, f_2)}, $$
for given  radial functions  $f_1, f_2$ in $\Sl(\C^d)$. The method consists in writing conditions for $ T_{f_1} T_{g}$ and $ T_{f_1} T_{g}$ to coincide on $\C\left[z\right]$, the space of all holomorphic polynomials. Precisely, we have 
\begin{equation}\label{equality e nu e kappa}
	\left\langle T_{f_1} T_{g} e_{\nu}, e_{\kappa}\right\rangle_{1,s}=\left\langle T_{g} T_{f_2} e_{\nu}, e_{\kappa}\right\rangle_{1,s},\ \text{ for all multiindices }\nu, \kappa.
\end{equation}

\subsection{The Mellin transform}
Relation (\ref{equality e nu e kappa}) will involve the  Mellin transform of functions related to the symbols and the function 
\begin{equation}\label{def phi} 
	\phi(r):=r^{2s}e^{- r^{2m}},\ r>0,
\end{equation}
which defines the density $d\mu_{m,1,s}$.

First recall some well known facts about the  Mellin transform \cite{PaKa}.

Let $f:\R_+\rightarrow\C$ be a function such that $f(x) x^{\Re \zeta-1}$ is integrable with respect to the Lebesgue measure on the real axis, for  $\zeta$ in the strip $$\text{St}(a,b): a< \Re \zeta<b,$$
 $a$ and $b$ being fixed real numbers. The Mellin transform,  defined by 
\begin{equation*}\label{mellin def} 
	\Ml\left[f\right](\zeta):=\int^{+\infty}_{0} f(x) x^{ \zeta-1}dx,
\end{equation*}
 is holomorphic in $\text{St}(a,b)$.

The Mellin convolution $f*g$ of two functions $f,g:\R_+\rightarrow\C$ is defined by 
\begin{equation*}\label{mellin conv} 
	f *g(x):=\int^{+\infty}_{0} f(y) g\left(\frac{x}{y}\right)\frac{dy}{y},
\end{equation*}
for all $x>0$ such that the integral exists. Under certain assumptions, we have 
$$ \Ml\left[f *g\right] = \Ml\left[f\right]\Ml\left[g\right]. $$

Integration in polar coordinates shows that the moments are related to the  Mellin transform of $\phi$ (see (\ref{def phi}))
\begin{equation*}\label{moment mellin} 
	\int^{+\infty}_{0}r^{2s+2d}e^{- r^{2m}}r^{2 \zeta-1}dr=\Ml\left[r^{2d}\phi(r)\right](2\zeta)=\frac{1}{2m}\Gamma\left(\frac{d+s+ \zeta}{m}\right).
\end{equation*}

We shall use the following properties.

\begin{prop}\label{Mu(k)=0} (\cite{BauerLee}, $\text{Proposition 4.11}$). Let $u\in\Al(\R_+)$, $a\in(0,2]$ and a  fixed integer $k_0\in\N$. If 
$$ \Ml\left[u(t) e^{-t}\right](ak+1)= \int^{+\infty}_{0} u(t) e^{-t} t^{ak}dt=0,$$
for all integers $k\geq k_0,$ then $u= 0$ a.e. on $\R_+$.
\end{prop}

For $u$ a function in $\Al(\R_+)$, we set 
$$f_{m,u}(x)=u(x) e^{-x^{2m}}.  $$

\begin{prop}\label{convo fu fv}
Let $u,v$ be in $\Al(\R_+)$.

The Mellin convolution $f_{m,u} *f_{m,v}$ exists on $(0,+\infty)$, and there exists a function $h_1$ in $\Al(\R_+)$, such  that 
$$ f_{m,u} *f_{m,v}(x)=h_1(x) e^{-x^{m}},\ x>0.  $$
If, in addition, $\text{supp }v\subset [0,1]$, there is a function $h_2$ in $\Al(\R_+)$, such  that 
$$  f_{m,u} *f_{m,v}(x)=h_2(x) e^{-x^{2m}},\ x>0.  $$
\end{prop}

\begin{proof}
This result has been proved in \cite{BauerLe} for $m=1.$ Performing a change of variable in the Mellin convolution, and setting $\xi=x^m$, we see that
\begin{align*}
 f_{m,u} *f_{m,v}(x)
=\int^{+\infty}_{0}\tilde u(t) e^{-t^2}\tilde v\left(\frac{\xi}{t}\right)\exp\left[-\frac{\xi^{2}}{t^{2}}\right]\frac{dt}{t}
=f_{1,\tilde u}*f_{1,\tilde v}(\xi),
\end{align*}
where the functions $$\begin{cases}
\tilde u(t):=\frac{1}{m}u\left(t^{1/m}\right),\\
\tilde v(t):=v\left(t^{1/m}\right)
\end{cases}$$
are in $\Al(\R_+)$. From \cite{BauerLe}, there is a function $\tilde h_1$ in $\Al(\R_+)$ satisfying
$$ f_{1,\tilde u}*f_{1,\tilde v}(\xi)= \tilde h_1(\xi) e^{-\xi}.$$
Putting $h_1(x)= \tilde h_1(x^m)$, we get the desired conclusion.

When  $\text{supp }v\subset [0,1]$, the reasoning is similar.
\end{proof}

Proposition \ref{periodic omega} extends Proposition 2.3 in \cite{BauerLe}  to general  $m\geq 1$, $s \geq 0$. We  just sketch its proof.

\begin{prop}\label{periodic omega}
Let $u:\R_+\rightarrow\C$ be a function such that $u\circ \left|.\right|$ is in $\Sl(\C^d)$, and 
$$ \psi(\zeta)=\frac{\Ml\left[u(r) r^{2s}e^{-r^{2m}}\right]\left(2\zeta+2\right)}{\Gamma\left(\frac{\zeta+1}{m}\right)} ,\ \zeta\in \Pi, $$
extends to a periodic function of period $j\in\N\setminus\left\{0\right\}$. Then $u$ is a constant function.
\end{prop}

\begin{proof}
By the change of variable $t=r^m$, we get
\begin{align*}
\Ml\left[u\phi\right]\left(2\zeta+2\right)&=\int^{+\infty}_{0}u(r)r^{2s}e^{-r^{2m}}r^{2\zeta+1}dr\\
&=\frac{1}{m}\int^{+\infty}_{0}u\left(t^{1/m}\right)t^{\frac{1}{m}\left(2s+2\zeta+2\right)-1}e^{-t^2}dt\\
&=\Ml\left[\tilde u\phi_G\right]\left(\frac{2\zeta+2}{m}\right),
\end{align*}
where $\tilde u(t):=\frac{1}{m}u\left(t^{1/m}\right)t^{\frac{2s}{m}}$ and $\phi_G(t):=e^{-t^2}.$
Since $u\circ \left|.\right|$ is in $\Sl(\C^d)$, we can choose $c\in(0,1)$ and $C>0$, such that $\left|\tilde u(t)\right|\leq C e^{ct^2}. $

By a change of variable, $\psi$ satisfies the estimate
$$ \left| \psi(\zeta)\right|\leq\frac{C}{2}\left(1-c\right)^{-\frac{\Re \zeta+1}{m}}\frac{\Gamma\left(\frac{\Re\zeta+1}{m}\right)}{\left|\Gamma\left(\frac{\zeta+1}{m}\right)\right|},\ \Re\zeta>-1,$$
 (see \cite{BauerLe} and \cite{BauerLee} for details).
Since $\left(1-c\right)^{-\frac{\Re \zeta+1}{m}}$ is bounded in any strip $a\leq \Re \zeta \leq a+j$, 
 the properties of the Gamma function and the arguments in \cite{BauerLee,BauerLe} imply that 
$$ \left| \psi(\zeta)\right|\leq C_1 e^{\frac{\pi}{2m}\left|\zeta\right|},\ \zeta\in \C. $$
Thus  $\psi$ is an entire function of exponential type which is periodic on the real axis. By Theorem 6.10.1 in \cite{boas},  $  \psi$ has the form
\begin{equation*}\label{psi fourier}
	\psi(\zeta)=\sum_{\left|k\right|\leq\frac{j}{4m}}a_k e^{2\pi i \frac{k}{j}\zeta}, \ \zeta\in \C.
\end{equation*}
This yields
\begin{equation*}\label{psi fourier2}
\Ml\left[\tilde u\phi_G\right] \left(\frac{2\zeta+2}{m}\right)=\Gamma\left(\frac{\zeta+1}{m}\right)\sum_{\left|k\right|\leq\frac{j}{4m}}a_k e^{2\pi i \frac{k}{j}\zeta}, \ \Re\zeta>-1. 
\end{equation*}
It is shown in \cite{BauerLee} that, if $\lambda\in\C$ and $\Re\lambda<1,$ 
$$ \Ml\left[e^{\lambda t^2}\phi_G\right] (z)=\frac{1}{2\pi}\left(1-\lambda\right)^{-\frac{z}{2}}\Gamma\left(\frac{z}{2}\right). $$
We now set 
$$ 
1-\lambda_k:=e^{-2\pi i \frac{k}{j}m},\ \text{ if }\left|k\right|\leq\frac{j}{4m}.
 $$

For $\left|k\right|\leq\frac{j}{4m}$, there are complex constants $b_k$,  such that when $\Re\zeta>-1,$
\begin{equation*}\label{psi mellin}
\Gamma\left(\frac{\zeta+1}{m}\right)\sum_{\left|k\right|\leq\frac{j}{4m}}a_k e^{2\pi i \frac{k}{j}\zeta}=\Ml\left[b_0 e^{- t^2} +\sum_{0\neq \left|k\right|\leq\frac{j}{4m}}b_k e^{\left(\lambda_k -1\right) t^2} \right] \left(\frac{2\zeta+2}{m}\right).
\end{equation*}
The Mellin transform being injective, we deduce that 
$$ \tilde u(t)=b_0  +\sum_{0\neq \left|k\right|\leq\frac{j}{4m}}b_k e^{\lambda_k  t^2} .$$
Since $\tilde u(t)e^{ct^2}$ is bounded for all $0<c<1$, we conclude that $\tilde u\equiv b_0 $.
\end{proof}

Recall that  the Beta function satisfies
\begin{equation*}\label{beta}
B(z,w):=\int^{1}_{0}x^{z-1}(1-x)^{w-1}dx=\frac{\Gamma(z) \Gamma(w)}{\Gamma(z+w)}, \ \text{for }\Re z, \Re w>0. 
\end{equation*}
 Therefore the quotient of two Gamma functions is a Mellin transform.

\begin{lemma}\label{quotient gamma}
Let $m\geq 1$. For $a,b$ be real numbers such that $a<b$, set $a'=a/m$ and $b'=b/m$. We have 
$$ \Phi(z):=\frac{\Gamma\left(\frac{a+z}{m}\right)}{\Gamma\left(\frac{b+z}{m}\right)}=\Ml\left[v\right](2z),\  z\in\C,\ \Re z>-a,$$
for some function $v$  in $\Al(\R_+)$, with $\text{supp }v\subset   [0,1]$.
\end{lemma}

\begin{proof}
Using a change of variable, we see that
\begin{align*}
\Phi(z)&=\frac{\Gamma\left(a'+\frac{z}{m}\right)}{\Gamma\left(b'+\frac{z}{m}\right)}
=\left[\Gamma\left(b'-a'\right)\right]^{-1}B\left(a'+\frac{z}{m},b'-a'\right)\\
&=\left[\Gamma\left(b'-a'\right)\right]^{-1}\int^{1}_{0}r^{a'+\frac{z}{m}-1}(1-r)^{b'-a'-1}dr
=\Ml\left[v_1\right](z),
\end{align*}	
where $v_1(x):=m\left[\Gamma\left(b'-a'\right)\right]^{-1}x^{ma'}(1-x^m)^{b'-a'-1}\chi_{(0,1)}(x).$
With the notation	$v(r):=2v_1(r^2),$ it is straightforward that
$$ \Phi(z) =\int_{0}^{+\infty}v(r) r^{2z-1}dr.$$
																																																	
\end{proof}
Let us also recall a result extracted from the proof of Proposition 3.1 from \cite{BauerLe}. Let $q $ be a positive integer. For any polynomial $p(\zeta)$ of degree $<q$, there is polynomial $\tilde p$ such that
\begin{equation}\label{par frac}
\frac{p(\zeta)}{(\zeta+1)\cdots(\zeta+q)}=\Ml\left[\tilde p \chi_{ [0,1]}\right](2\zeta+2),\ \zeta\in\Pi.
\end{equation}

\subsection{Toeplitz operators with radial symbols}
For $z=(z_1,\cdots,z_d)\in\C^d$, we set
$$ \Sigma z=z_1+\cdots+z_d.$$
If $f$ is  a radial function in $\Sl(\C^d)$, then $T_f$ is diagonal with respect to the orthonormal basis $\left(e_{\nu}\right)_{\nu}$. Since the measure $d\mu_{m,1,s}$ is radial,  the eigenvalue  $\omega(f,\nu)$ associated to $e_{\nu}$ depends only on $\Sigma\nu$
 for $\nu\in\N^d.$
  Indeed, from Lemma 2.2 in  \cite{BauerLee}, we know  that
$$ T_f e_{\nu}=\omega(f,\nu) e_{\nu},$$
and 
\begin{equation*}\label{omega |nu|}
	\omega(f,\nu) =\frac{2m}{\Gamma\left(\frac{d+s+\Sigma\nu}{m}\right)}\int^{+\infty}_{0}r^{2d+2s}f(r)e^{-r^{2m}}r^{2\Sigma\nu-1}dr .
\end{equation*}
The function
\begin{equation*}\label{Omega |nu|}
	\Omega(f,\zeta) :=\frac{ \Ml\left[2m f(r)r^{2d+2s}e^{-r^{2m}}\right]\left(2 \zeta\right) }{\Gamma\left(\frac{d+s+\zeta}{m}\right)} , \ \text{ for }\zeta\in\C,\ \Re\zeta>-s-d,
\end{equation*}
 is holomorphic on its domain, and 
\begin{equation}\label{omega |nu|}
\omega(f,\nu) =\Omega(f,\Sigma\nu)=\frac{1}{S_{1,s}(\nu)}\Gl f(\nu),\ \text{ for }\nu\in\N^d.
\end{equation}

Here 
\begin{equation}\label{def Gl}
\Gl g(z):=\int_{\C^d}g(x)\left|x_1\right|^{2 z_1} \cdots \left|x_d\right|^{2 z_d}d\mu_{m,1,s}(x),\ z\in\Pi^d,
\end{equation}
for a measurable function $g$ such that the integral is defined. The transform $\Gl$ has been introduced in \cite{BauerLe} when $m=1, s=0.$ The properties shown in the case of the Gaussian measure extend to the measure $d\mu_{m,1,s}$, $m\geq 1$, $s\geq 0$.

When $d=1$, $\Gl g$ is expressed in terms of the radialization of $g$ 
\begin{equation*}\label{radialization}
g_{\text{rad}}(r):=\frac{1}{\pi}\int^{2\pi}_{0}g(r e^{i\theta})d\theta,\ r>0.
\end{equation*}
 We have 
\begin{equation}\label{Gl radialization}
\Gl g(z)=\Ml\left[\frac{m}{\Gamma\left(\frac{1+s}{m}\right)}g_{\text{rad}}(r)r^{2s+2}e^{-r^{2m}}\right](2z),\ \text{ for }z\in \Pi.
\end{equation}
If $g$ is in $\DD_c$ for some $c<1$, the function $g(z) e^{-c\left|z\right|^{2m}}$ is bounded  on $\C^d$; the Dominated Convergence  Theorem and Morera's Theorem imply that $\Gl g$ is analytic on $\Pi^d$, and continuous on $\overline\Pi^d$ . 

For technical reasons, it is easier to deal with functions having polynomial growth at infinity.  For $t$  a positive number, we define the transform
$$V_t g(x):=g(tx) e^{(1-t^{2m})\left|x\right|^{2m}},\ x\in\C^d,  $$
for any function $g:\C^d\rightarrow\C$.  

Assume that $g$ is in $\Sl(\C^d)$;
if $t$ is large enough, $V_t g(x)$ tends to $0$ as $\left|x\right|\rightarrow\infty$; in particular,  $V_t g$ belongs to $\Al(\C^d)$ .

A change of variable shows that, whenever the integral  (\ref{def Gl})    is convergent, we have
\begin{equation}\label{GlVt}
 \Gl V_t g(z)=t^{-2(s+d)}t^{-2\Sigma z} \Gl  g(z).  
\end{equation}

\subsection{The equation $T_g T_{f_1}=T_{f_2}T_g$}\label{toep rad}
Assume that $f_1, f_2$ are radial functions in $\Sl(\C^d)$. If $g$ is  also in $\Sl(\C^d)$, the products $T_g T_{f_1}$ and $T_{f_2}T_g$ are densely defined on $\FF$. We now write conditions for $T_g T_{f_1}$ and $T_{f_2}T_g$ to coincide on $\C\left[z\right]$. For multiindices $\kappa,\nu$, 
we have
\begin{align*}
  \left\langle T_{f_2}T_g e_{\kappa}, e_{\nu}\right\rangle_{\FF}&=
 \left\langle P_{m,1,s}( g e_{\kappa}), \overline f_2e_{\nu}\right\rangle_{\FF}= \left\langle  g e_{\kappa}, T_{\overline f_2}e_{\nu}\right\rangle_{\FF}\\
&=\left\langle g e_{\kappa}, \omega(\overline f_2,\nu)e_{\nu}\right\rangle_{\FF}
= \Omega(f_2,\Sigma\nu) \left\langle g e_{\kappa}, e_{\nu}\right\rangle_{\FF}
\end{align*}
and 
$$ \left\langle T_g T_{f_1}e_{\kappa}, e_{\nu}\right\rangle_{\FF}= \Omega( f_1,\Sigma\kappa) \left\langle g e_{\kappa}, e_{\nu}\right\rangle_{\FF} .$$
Thus,  $T_g T_{f_1}=T_{f_2}T_g$ on $\C\left[z\right]$ if and only if
\begin{equation*}\label{eq 1}
0=\left[\Omega( f_1,\Sigma\kappa)-\Omega( f_2,\Sigma\nu)\right]\int_{\C^d}g(x)x^{\kappa}  \overline x^{\nu}d\mu_{m,1,s}(x), \text{  for all }\kappa,\nu\in\N^d.
\end{equation*}
Fix $k,n$  in  $\N^d$, such that $\Sigma k<\Sigma n$, and set 
 $$ g_1(x)=g_2(x):=g(x)x^k \overline x^n ,\ x\in\C^d.$$
 Replacing $\kappa$ by $k+l$ and $\nu$ by $n+l$,  we  obtain
\begin{equation}\label{eq 2}
0=\left[\Omega( f_1, \Sigma k+\Sigma l)-\Omega( f_2,\Sigma n+\Sigma l)\right]\Gl g_1(l), \text{  for all }l\in\N^d.
\end{equation}

Formulas (\ref{s alpha}) and (\ref{omega |nu|}) imply that
\begin{align*}
A&:=C^{-1}_{s} \left[\Omega( f_1, \Sigma k+\Sigma l)-\Omega( f_2,\Sigma n+\Sigma l)\right]\\
&=\frac{\Gamma\left(d+\Sigma(k+l)\right)\Gl \tilde f_1(l)}{(k+l)!\Gamma\left(\frac{d+s+\Sigma(k+l)}{m}\right)}-\frac{\Gamma\left(d+\Sigma(n+l)\right)\Gl \tilde f_2(l)}{(n+l)!\Gamma\left(\frac{d+s+\Sigma(n+l)}{m}\right)},
\end{align*}
where
\begin{align*}
\tilde f_1(z)&= f_1(z)\left|x_1\right|^{2k_1}\cdots\left|x_d\right|^{2k_d},\\
\tilde f_2(z)&= f_2(z)\left|x_1\right|^{2n_1}\cdots\left|x_d\right|^{2n_d},\ z\in\C^d. 
\end{align*}
Set
$$ \Delta(l):=\frac{(k+n+l)!}{(k+l)!}\frac{\Gamma\left({d+\Sigma(k+l)}\right)}{\Gamma\left({d+\Sigma(n+l)}\right)}\Gl \tilde f_1(l)-
\frac{(k+n+l)!}{(n+l)!}\frac{\Gamma\left(\frac{d+s+\Sigma(k+l)}{m}\right)}{\Gamma\left(\frac{d+s+\Sigma(n+l)}{m}\right)}\Gl \tilde f_2(l) .$$
Hence (\ref{eq 2}) is equivalent to 
\begin{equation}\label{G(l)=0}
G(l):= \Delta(l)\Gl g_1(l)=0,\ \text{ for all }l\in\N^d.
\end{equation}
Observe that $\frac{(k+n+l)!}{(n+l)!}$ and $\frac{(k+n+l)!}{(k+l)!} $ are polynomials in $l$.
 In fact, there exist holomorphic polynomials $p_1$ and $p_2$ such that  $\left\{\Delta(l),\ l\in\N^d\right\}$ and $\left\{G(l),\ l\in\N^d\right\}$ are the restrictions to $\N^d$ of the holomorphic functions
\begin{equation*}\label{def0 delta}
\Delta(z):=p_1(z)\frac{\Gamma\left(\frac{d+s+\Sigma(k+z)}{m}\right)}{\Gamma\left(\frac{d+s+\Sigma(n+z)}{m}\right)}\Gl \tilde f_1(z)-p_2(z)\frac{\Gamma\left({d+\Sigma(k+z)}\right)}{\Gamma\left({d+\Sigma(n+z)}\right)}\Gl \tilde f_2(z) ,
\end{equation*}
and 
\begin{equation*}\label{def0 G}
G(z):=\Delta(z)\Gl g_1(z),\ \text{ for } z\in \Pi^d,
\end{equation*}
respectively.
From Lemma \ref{quotient gamma},  there are functions $v_k$, $k=1,2$, in $\Al(\R_+)$ such that, for all $z\in \Pi^d$,
\begin{equation}\label{def G}
	G(z)= \sum^{2}_{j=1}p_j(z) \Ml\left[v_j\chi_{ [0,1]}\right](2\Sigma z) \Gl \tilde f_j(z)\Gl g_j(z).
\end{equation}

The next Proposition  proves that equation (\ref{G(l)=0}) implies that $G$ is identically zero.

\begin{prop}\label{Gl=0 implies G=0}
Let $N$ be a positive integer. 

Suppose that for every $j=1..N$, $p_j$ is a holomorphic polynomial, $v_j$ is  in $\Al(\R_+)$, and $ f_j,g_j$ are in $\Sl(\C^d)$. In addition, we assume that the Mellin transform $\Ml\left[v_j\chi_{ [0,1]}\right]$ is holomorphic on $\Pi$, continuous on $\overline\Pi$. Set
$$ G(z):= \sum^{N}_{j=1}p_j(z) \Ml\left[v_j\chi_{ [0,1]}\right](2\Sigma z) \Gl  f_j(z)\Gl g_j(z),\ z\in\Pi^d. $$
If $G(l)=0$ for all $l\in\N^d$, then $G$ identically vanishes on  $\Pi^d. $
\end{prop}
Our proof of Proposition \ref{Gl=0 implies G=0}, is inspired  of \cite{BauerLe}. However, some extra work is needed.
In particular,  there is no simple identity between $e^{-\left|z\right|^{2m}}$ and $e^{-\left|z_1\right|^{2m}}\times e^{-\left|z'\right|^{2m}}$, where $z=(z_1,z')\in\C^d.$

We begin with a technical Lemma.

\begin{lemma}\label{I sigma}
Let $\sigma> 0$ and $q\in \N$. For  $x=(x_1,\cdots,x_q)\in\R^{q}_{+} $ and $r\geq 0$, define 
$$I_{\sigma}(r,x):=\int_{\C^{q}}\left(1+(r^2+\left|w\right|^2)^{ \sigma}\right)\left|w_1\right|^{2 x_1}\cdots \left|w_q\right|^{2 x_q}(r^2+\left|w\right|^2)^{s}e^{-\left(r^2+\left|w\right|^2\right)^{m}}dv(w).$$
Then, 
$$ I_{\sigma}(r,x)\leq A(x) (r^2+ \sigma+s) ^{\sigma+s},\ \text{ for all }r\geq 0, $$ 
where $A(x) $ depends on $x$.
\end{lemma}

\begin{proof}
First we shall estimate $I_{\sigma}(r,x)$ when $r\geq 1$. For $\tau \geq 0$,
  set 
$$  C_{\tau}(t,\rho):=\left(\rho+t\right)^{\tau}e^{-\left(\rho+t\right)^{m}},\ t>0,\ \rho\geq 0.$$
Putting
$$ b_{\tau} (t,\rho):=\left(\rho+t\right)^{\tau}e^{- m\rho t^{m-1}},$$
and using Taylor's Theorem, we have 
\begin{equation}\label{C tau}
	  C_{\tau}(t,\rho)\leq  b_{\tau}  (t,\rho) e^{-  t^{m}},\ \text{ for }\rho\geq 1.
\end{equation}

Obviously, $ b_{\tau} (t,\rho)\leq\left(\rho+t\right)^{\tau}$.
Besides,
the behavior of  the function $\rho\rightarrow b (t,\rho) $ shows that
\begin{equation}\label{b tau}
 b_{\tau} (t,\rho)\leq
b_{\tau}(t,0)\text{ if }t\geq\tau.
\end{equation}
Observing that, when $\rho\geq 1$, $C_s(t,\rho)\leq C_{s+\sigma}(t,\rho)$, and using (\ref{C tau}) and (\ref{b tau}) we obtain that, for $r\geq 1$,
\begin{align*}
I_{\sigma}(r,x)
&\leq 2 \int_{\C^q}|w_1|^{2 x_1}\cdots  |w_q|^{2 x_q} b_{s+\sigma}\left(\left|w\right|^2, r^2\right)e^{-\left|w\right|^{2m}}dv(w)
\leq 2(I_1+I_2),
\end{align*}
where 
$$ I_1 =\int_{\left|w\right|^2\leq s+\sigma}|w_1|^{2 x_1}\cdots  |w_q|^{2 x_q}(\left|w\right|^2+ r^2)^{s+\sigma}e^{-\left|w\right|^{2m}}dv(w)$$
and 
$$ I_2 =\int_{\left|w\right|^2> s+\sigma}|w_1|^{2 x_1}\cdots  |w_q|^{2 x_q}\left|w\right|^{2(s+\sigma)}e^{-\left|w\right|^{2m}}dv(w).$$
Notice that 
$$ I_1\leq ( s+\sigma + r^2)^{s+\sigma}\int_{\left|w\right|^2\leq s+\sigma}|w_1|^{2 x_1}\cdots  |w_q|^{2 x_q}e^{-\left|w\right|^{2m}}dv(w),$$
and 
\begin{align*}
 I_2
&\leq ( s+\sigma + r^2)^{s+\sigma}\int_{\left|w\right|^2> s+\sigma} |w_1|^{2 x_1}\cdots  |w_q|^{2 x_q}  \left|w\right|^{2  ( s+\sigma)} e^{-\left|w\right|^{2m}}dv(w).
\end{align*}
Setting 
\begin{multline*}
 A_1(x):= \int_{\left|w\right|^2\leq s+\sigma}|w_1|^{2 x_1}\cdots  |w_q|^{2 x_q}e^{-\left|w\right|^{2m}}dv(w) \\+\int_{\left|w\right|^2> s+\sigma} |w_1|^{2 x_1}\cdots  |w_q|^{2 x_q}  \left|w\right|^{2  ( s+\sigma)} e^{-\left|w\right|^{2m}}dv(w),
\end{multline*}


we have 

\begin{equation}\label{r geq 1}
	I_{\sigma}(r,x)\leq 2 ( s+\sigma + r^2)^{s+\sigma} A_1(x),\ r\geq 1.
\end{equation}
Besides, the continuity of   $I_{\sigma}(.,x)$ implies that there exists $A_2(x)$ such that
\begin{equation}\label{r leq 1}
	I_{\sigma}(r,x)\leq  A_2(x)
	,\ \text{ for }r\leq 1.
\end{equation}
Combining (\ref{r geq 1}) and (\ref{r leq 1}), we obtain the Lemma.

\end{proof}


\begin{proof}[Proof of Proposition \ref{Gl=0 implies G=0}]
Let $t>0$. Relation  (\ref{GlVt}) implies that
\begin{align*}
G_t(z)&:=\sum^{N}_{j=1}p_j(z) \Ml\left[v_j\chi_{ [0,1]}\right](2\Sigma z) \Gl V_t f_j(z)\Gl V_t g_j(z),\\
&=t^{-4(s+d)}t^{-4\Sigma z} G(z).
\end{align*}
This shows that the holomorphic functions  $G$ and $G_t$ have exactly the same zeroes on $\Pi^d$. Recall that, for large $t$, $V_t f_j$ and $V_t g_j$  in $\Al(\C^d)$. Consequently, we  may assume that the functions $f_j, g_j$ are in $\Al(\C^d)$. In particular,  
$G$ is holomorphic on $\Pi^d$, and continuous on $\overline \Pi^d$.\\ 

First,  suppose that $d=1$. Let $q$ be an integer which is strictly larger than $\max\left\{\text{ deg }(p_j),\ j=1..N\right\}$.
From   relations (\ref{par frac}) and (\ref{Gl radialization}), there are polynomials $q_j$ and functions $ \tilde f_j,\tilde g_j$ in $\Al(\R_+)$, such that, for $\zeta\in \Pi,$ 
\begin{align*}
Q(\zeta)&:=\frac{G(\zeta)}{(\zeta+1)\cdots(\zeta+q)}=\sum^{N}_{j=1}\frac{p_j(\zeta)}{(\zeta+1)\cdots(\zeta+q)}\Ml\left[v_j\chi_{ [0,1]}\right](2 \zeta) \Gl  f_j(\zeta)\Gl g_j(\zeta)\\
&=\sum^{N}_{j=1}\left(\Ml\left[ q_j\chi_{ [0,1]}\right]\Ml\left[v_j\chi_{ [0,1]}\right]\Ml\left[\tilde f_j r^{2s+2}e^{-r^{2m}}\right]\Ml\left[\tilde g_j r^{2s+2}e^{-r^{2m}}\right]\right)(2 \zeta).
\end{align*}
Due to the properties of the Mellin convolution and Proposition \ref{convo fu fv}, there exist functions $h_j$ 
in $\Al(\R_+)$, such that
\begin{align*}
Q(\zeta)&=\sum^{N}_{j=1}\Ml\left[ q_j\chi_{ [0,1]}*v_j\chi_{ [0,1]}*\tilde f_j r^{2s+2}e^{-r^{2m}}*\tilde g_j r^{2s+2}e^{-r^{2m}}\right](2\zeta)\\
&=\sum^{N}_{j=1}\Ml\left[ h_j e^{-r^{m}}\right](2\zeta).
\end{align*}
Putting 
$  H:=\sum^{r}_{j=1} h_j,$ 
we get
$$ G(\zeta)=(\zeta+1)\cdots(\zeta+q)\Ml\left[H e^{-r^{m}}\right](2\zeta),\ \zeta\in\Pi. $$
The function  $G$ vanishes on $\N$. Then, for all $k$ in $\N$,
\begin{align*}
0&=\int^{+\infty}_{0}H(r) e^{-r^{m}} r^{2k-1}dr
=\int^{+\infty}_{0}\frac{1}{m}H(t^{1/m})e^{-t} t^{\frac{2}{m}k-1}dt.
\end{align*}
 Proposition \ref{Mu(k)=0} with $a=2/m\leq 2,$ enables us to conclude that $H= 0$ a.e. on $\R_+$. Hence, $G= 0$ a.e. on $\Pi$. \\

In order to handle the general case $d\geq 2$, we shall use the notation $z=(z_1,z')$, $w=(w_1,w')$, where $z'=(z_2,\cdots,z_d)$ , $ \Re z'=x'=(x_2,\cdots,x_d)$, and  $w'=(w_2,\cdots,w_d)$. 

Let $f$ be in $\Al(\C^d)$. For $z$ in $\C^d$, integration in polar coordinates with respect to the variable $w_1=r e^{i\theta}$ gives
\begin{align*}
\Gl f(z_1,z')&=\int^{+\infty}_{0} F(r,z')r^{2\Re z_1+1}dr
=\Ml\left[F(.,z')\right](2z_1+2),
\end{align*}
where 
$$ \frac{F(r,z')}{c_{m,1,s}} :=\int_{\C^{d-1}\times [0,2\pi]}f(r e^{i\theta},w') \left|w_2\right|^{2 z_2}\cdots \left|w_d\right|^{2 z_d}(r^2+\left|w'\right|^2)^{s}e^{-\left(r^2+\left|w'\right|^2\right)^{m}}dv(w')d\theta.$$
We now fix $(r,z')$ in $\R_+\times\C^{d-1}$. Since $f$ is in $\Al(\C^d)$, there exists $\sigma> 0$ such that
\begin{equation}\label{ineq F}
	 \left|  F(r,z') \right|
\lesssim I_{\sigma}(r,x'),
\end{equation}
where $ I_{\sigma}(r,x')$ has been estimated  in Lemma \ref{I sigma}. Therefore, there exists $A(x')$, which depends on $x'$, such that
 \begin{equation*}\label{ineq F 2}
	\left|F(r,z')\right|\lesssim  A(x')(r^2+\sigma+s)^{\sigma+s}, \ (r,z') \in \R_+\times\Pi^{d-1},
\end{equation*}
and $F(.,z')$ is in $\Al(\R_+)$.
We next write $G(z_1,z')$ as a Mellin transform at $z_1\in\Pi
$. For $j=1..N$, set
$$ u_j(r,z'):=(v_j\chi_{[0,1]})(r)r^{2\Sigma z'},\ r>0. $$
Therefore,
\begin{align*}
G(z_1,z')&=\sum^{N}_{j=1}p_j(z) \Ml\left[v_j\chi_{ [0,1]}\right](2 z_1+2\Sigma z') \Gl  f_j(z_1,z')\Gl g_j(z_1,z')\\
&=\sum^{N}_{j=1}p_j(z_1,z') \Ml\left[u_j(.,z')\chi_{ [0,1]}\right](2z_1) \Ml \left[ F_j(.,z')\right](2z_1)\Ml \left[ G_j(.,z')\right](2z_1),
\end{align*}
where the functions $F_j$ and $G_j$, analytic in the variable $z'\in \Pi^{d-1},$ are of polynomial growth with respect to $r>0$.

 By assumption, $G(k_1,k')=0$ for all $(k_1,k')$ in $\N\times\N^{d-1}$. From the case $d=1$, we see that  $G(z_1,k')=0$ for all $(z_1,k')$ in $\Pi\times\N^{d-1}$. Now, integrating with respect to $w_2$ in polar coordinates, we  express $G(z)$ as a Mellin transform, in the variable $z_2$, namely
$$  G(z)=\sum^{N}_{j=1}p_j(z) \left(\Ml\left[\tilde u_j(.)\right] \Ml \left[ \tilde F_j(z_1,.,z_3,\cdots,z_d)\right]\Ml \left[ \tilde G_j(z_1,.,z_3,\cdots,z_d)\right]\right)(2z_2). $$ 

Here,  $\tilde F_j$, $ \tilde G_j$ are of polynomial growth with respect to $r>0$,
 and 
$$\tilde u_j(r):= \left(v_j\chi_{ [0,1]}\right)(r) r^{2 z_1+2 z_3+\cdots+  2 z_d}.$$

For fixed $(z_1,k_3,\cdots,k_d)\in \Pi\times\N^{d-2} $, we have $G(z_1,k_2,k_3,\cdots,k_d)=0$, for all $k_2$ in $\N$, and the case $d=1$ again implies that $G(z_1,z_2,k_3,\cdots,k_d)=0 $, for all  $(z_1,z_2,k_3,\cdots,k_d)$in $\Pi^2\times\N^{d-2}$ . The proof is obtained by induction.
\end{proof}
Recall that equation $\El(f_1,f_2)$ is equivalent to 
$$ 0=\left[\Omega( f_1, \Sigma k+\Sigma l)-\Omega( f_2,\Sigma n+\Sigma l)\right]\Gl\left(M_{z^k \overline z^n}g\right) (l),\ \text{ for all }l,n,k\in\N^d. $$

 Let $n,k\in\N^d.$  Proposition \ref{Gl=0 implies G=0} implies that
$$ 0=\left[\Omega( f_1, \Sigma k+\Sigma z)-\Omega( f_2,\Sigma n+\Sigma z)\right]\Gl\left(M_{z^k \overline z^n}g\right)(z), \ \text{ for all }z\in\Pi^d.$$
By analyticity, this means that, either 
\begin{equation*}\label{Omega=0}
	\Omega( f_1, \Sigma k+ \zeta )=\Omega( f_2,\Sigma n+\zeta) \ \text{ for all }\zeta\in\Pi, 
\end{equation*}
 or
 \begin{equation}\label{Fg=0}
\int_{\C^d}g(x)\left|x_1\right|^{2 z_1}\cdots \left|x_d\right|^{2 z_d}x^{k}  \overline x^{n}d\mu_{m,1,s}(x)=0 \ \text{ for all }z\in\Pi^d.
\end{equation}
In order to discuss these equations,  we  borrow some tools  from \cite{BauerLe} p. 2327-2628. Consider the set 
$$ \Zl(f_1,f_2):=\left\{n\in\Z,\ \Omega(f_1, \zeta )= \Omega(f_2, \zeta +n) \text{ for all }\zeta\in\Pi, \text{ with } \Re \zeta\text{ large enough}\right\}. $$
  
\begin{prop}\label{Z f1 f2}
Let $f_1,f_2$ be in $\Sl(\C^d)$.
\begin{enumerate}
	\item If $f_1,f_2$ are constant functions, then $\Zl(f_1,f_2)=\varnothing$ or $\Z$.
	
	\item If at least one of the functions $f_1,f_2$  is non constant, then $\Zl(f_1,f_2)=\varnothing$ or $\left\{q\right\}$ for some integer $q$.
\end{enumerate}
\end{prop}

\begin{proof}
The proof presented in \cite{BauerLe} verbatim extends to the setting of $F^{2}_{m,1,s}$, for general $m\geq 1,s\geq 0$. The argument essentially relies on the analyticity of $\Omega(f, . )$, and on the fact that, if $\Omega(f, . )$ has a period $j\in \N\setminus\left\{0\right\}$, then $f$ is a constant function. The latter fact has been proved in Proposition \ref{periodic omega}.
\end{proof}

On another side, equation  (\ref{Fg=0}) is handled via the following Lemma.
\begin{lemma}\label{invariance g}
Let $g$ be in $L^2(\C^d,d\mu_{m,1,s})$, $l$ be in $\Z^d$, and $q$ be an integer. Then the following are equivalent.
\begin{enumerate}
	\item $\int_{\C^d}g(x)x^{k}  \overline x^{n}d\mu_{m,1,s}(x)=0$  for all multiindices $n,k\in\N^d$ such that $(n-k).l\neq q$. Here we write $(n-k).l=(n_1-k_1)l_1+\cdots+  (n_d-k_d)l_d$.
	
	\item $g(\gamma^{l_1}z_1,\cdots,\gamma^{l_d}z_d)=\overline \gamma^q g(z)$ for a.a. $\gamma\in\T$ and $z=(z_1,\cdots,z_d)\in\C^d.$
\end{enumerate}
\end{lemma}

\begin{proof}
 We assume $m\geq 1$, $s\geq 0$. Since $d\mu_{m,1,s}$ is rotation invariant and the set of polynomials $\C\left[z, \overline z\right]=\text{ span }\left\{z^k \overline z^k,\ n,k\in\N^d\right\}$ is dense in $L^2(\C^d,d\mu_{m,1,s} )$ (Proposition \ref{density poly }), the reasoning   given in \cite{BauerLe} for the Gaussian measure ($m=1, s=0$) applies here.
\end{proof}

\section{Commuting problem}\label{comm}
For $f$ a radial function in $\Sl(\C^d)$, we want to find all those functions $g$ in $\Sl(\C^d)$ satisfying equation 
$$\El(f,f): \quad T_f T_g=T_g T_f.$$

\begin{proof}[Proof of Theorem A]
 Since $0\in \Zl(f,f)$, it follows from Section \ref{eq Toep} that $\Zl(f,f)=\left\{0\right\}$.

Let $n,k$ be in $\N^d$. In section \ref{toep rad}, we have seen that $\El(f,f)$ implies that either 
\begin{equation}\label{Omega f=0}
	\Omega( f, \Sigma k+ \zeta )-\Omega( f,\Sigma n+\zeta)=0, \ \text{ for all }\zeta\in\Pi, 
\end{equation}
 or (\ref{Fg=0}) holds.

If $\Sigma k\neq \Sigma n$, which means $(n-k).(1,\cdots,1)\neq 0$, (\ref{Omega f=0}) does not hold because $\Zl(f,f)=\left\{0\right\}$. We now apply  Lemma \ref{invariance g} with $q=0$.

\end{proof}

Now we give an example of a radial function $f$ which is not in $\Sl(\C^d)$, such that the commutant of $T_f$ contains non radial functions. For $\Re \lambda<1/2$, the function $f(z)=e^{\lambda\left|z\right|^{2m}}$ is the symbol of a diagonal Toeplitz operator.  Set $c:=\left(1-\lambda\right)^{-1/m}$. A direct computation shows that
$$  T_f z^{\kappa}= c^{d+s}\left(cz\right)^{\kappa},\ {\kappa}\in\N^d.$$
We choose $c=e^{2i\frac{\pi}{N}}$, with $N>6m$, which ensures that $\Re \lambda=1-\cos\left(\frac{2\pi m}{N}\right)<1-\cos\left(\frac{\pi }{3}\right)=\frac{1}{2}.$

We see that $T_f=e^{2i\frac{\pi(d+s)}{N}}V_N$ on $\C\left[z\right]$, where 
$$V_Nu(z)=u\left(e^{2i\frac{\pi}{N}}z\right),\ u\in\FF.$$
For any bounded function $g$, we have $V_NT_g V_{-N}=T_{V_N g}$. Therefore, the condition $V_Ng=g$ implies that $T_f$ and $T_g$ commute on $\C\left[z\right]$. For example, the non radial bounded function 
$$  g(z):= \frac{z^{N}_{1}}{\left|z\right|^{N}},\ z\in\C^d,$$
belongs to the commutant of $T_f.$

\section{Zero Product Problem}\label{zeroPb}
For $f$ a radial function in $\Sl(\C^d)$, we consider the equations 
$$\El(f,0):\quad T_f T_g=0,$$
and
$$\El(0,f) :\quad T_g T_f=0.$$

\begin{proof}[Proof of Theorem B]
If $f$ is a non zero constant function, the products $T_f T_g$ and $T_g T_f$ are constant multiples of $T_g.$ If $\El(f,0)$ or $\El(0,f)$ holds, then $g=0$ a.e. by Proposition \ref{one to one }.

Next, assume that $f$ is not a constant function. Let  $k$ be in $\N^d$. The discussion in Section \ref{toep rad} shows that we have either 
\begin{equation*}\label{omega f=0}
	\Omega( f, \Sigma k+ \zeta )=0 , \ \text{ for all }\zeta\in\Pi, 
\end{equation*}
or (\ref{Fg=0}). Since $\Zl(f,0)=\Zl(0,f)=\varnothing,$ then $\El(f,0)$ and $\El(0,f)$ are both equivalent to
$$ \int_{\C^d}g(x)\left|x_1\right|^{2 z_1}\cdots \left|x_d\right|^{2 z_d}x^{k}  \overline x^{n}d\mu_{m,1,s}(x)=0 \ \text{ for all }z\in\Pi^d, n,k\in \N^d. $$
The Theorem is proved by the density of polynomials in $L^2(\C^d,d\mu_{m,1,s} )$ (Proposition \ref{density poly }).
\end{proof}

Declarations of interest: none.

{\it Acknowledgements.} We would like to thank the referee for useful comments.

\end{document}